\documentclass[a4paper]{article}

\usepackage[top=3cm, bottom=3cm, left=3cm, right=3cm]{geometry}

\usepackage{graphicx}
\usepackage{bxtexlogo}
\usepackage{amsmath,amsthm,amssymb,mathtools,tensor}
\usepackage{cases}
\usepackage{color}
\usepackage{bm}
\usepackage{comment}
\usepackage{tikz}

\newtheorem{mainthm}{Theorem}
\newtheorem{theorem}{Theorem}[section]

\newtheorem{lemma}[theorem]{Lemma}
\newtheorem{cor}[theorem]{Corollary}
\theoremstyle{definition}
\newtheorem{rem}[theorem]{Remark}

\numberwithin{equation}{section}

\providecommand{\abs}[1]{\left\lvert#1\right\rvert}

\title{Isotropy invariant graphical mean curvature flows \\in warped products}

\author{Naotoshi Fujihara and Naoyuki Koike}
\date{}
\begin{document}
\maketitle
\begin{abstract}
In this paper, we study the graphical mean curvature flow in a warped product $_r G/K \times I$, where $G/K$ is a symmetric space of compact type, $I$ is an open interval, and $r$ is a smooth positive function on $I$. 
If the initial hypersurface is $K$-equivariant, then the $K$-equivariance is preserved along the mean curvature flow. Here, we note that isotropy group $K$ acts naturally on both $G/K$ and $_r G/K \times I$. 
If the flow is graphical, then it follows from the $K$-equivariance of the flow that it can be described by using $K$-invariant functions on $G/K$.
We derive the flow equation which these functions satisfy. 
By using the flow equation, we prove that the mean curvature flow exists for infinite time under the conditions that $G/K$ is a rank one symmetric space of compact type and the warping function $r$ satisfies certain additional properties. 
The proof is carried out by estimating the gradient of the $K$-invariant functions satisfying the flow equation.
\end{abstract}

\section{Introduction}\label{sec:intro}

Let $(M, g_M)$ be a closed Riemannian manifold and $(I, d z \otimes d z)$ an open interval. 
For smooth positive functions $\varphi \colon M \to \mathbb{R}$ and $r \colon I \to \mathbb{R}$, two types of warped product $M \times _{\varphi} I$ and $_{r} M \times I$ are defined respectively. 
In \cite{BM}, Borisenko and Miquel studied the mean curvature flow starting from graph hypersurfaces in the warped product $M \times _{\varphi} I$. In this case, the graph property is generally preserved along the flow. 
However, in the warped product $_{r} M \times I$, it is known that the graph property is generally not preserved along the mean curvature flow. 
The condition for the graph property to be preserved along the flow was given by Huang, Zhang, and Zhou \cite{HZZ}. They studied the asymptotic behavior of the graphical mean curvature flow in the warped product $_r M \times I$ of a closed Riemannian manifold $M$ with Ricci curvatures bounded from below and an open interval $I$ by a certain kind of warping function $r$. 
In \cite{F}, the graphical mean curvature flow in the warped product $_r M \times I$ was studied in a similar setting. 
In this paper, we study the behavior of the $K$-invariant graphical mean curvature flow in the warped product $_r G/K \times I$ of a symmetric space $G/K$ of compact type and an open interval $I$, and prove the ones similar to results in \cite{HZZ} and \cite{F}. 

Let $M$ be an $n$-dimensional smooth manifold and $(\overline{M}, \overline{g})$ an $(n+1)$-dimensional smooth Riemannian manifold.
Also let $F \colon M \times [0,T) \to \overline{M}$ be a smooth map such that $F_t \coloneqq F(\cdot, t) \colon M \to \overline{M}$ ($t \in [0,T)$) are smooth embeddings. 
The map $F$ is called a \textsl{mean curvature flow} if it satisfies 
\begin{equation}\label{eq:MCF}
\frac{\partial F}{\partial t} = \bm{H}_t,
\end{equation}
where $\bm{H}_t$ is a mean curvature vector field of the hypersurface $F_t(M)$. 
Let $G/K$ be a rank one symmetric space of compact type, and $\mathfrak{g} = \mathfrak{k} \oplus \mathfrak{p}$ be a canonical decomposition of the Lie algebra $\mathfrak{g}$ associated to the symmetric pair $(G, K)$. 
We take the $G$-invariant metric determined by $- B \vert_{\mathfrak{p} \times \mathfrak{p}}$ as a Riemannian metric on $G/K$, where $B$ is the Killing form of $\mathfrak{g}$. 
Let $\mathfrak{a}$ be a maximal abelian subspace of $\mathfrak{p}$ and $e^0_1$ be a unit vector of $\mathfrak{a}$.
Then, for a $K$-invariant function $u_0 \colon G/K \to I$, define a function $\widehat{u}_0 \colon \mathbb{R} \to I$ by
\begin{equation*}
\widehat{u}_0(x) \coloneqq u_0(\exp_{e K}(x e^0_1)),
\end{equation*}
and set $\widehat{u}'_0 \coloneqq d \widehat{u}_0 /d x$.
Let $F \colon G/K \times [0,T) \to {_r}G/K \times I$ be the mean curvature flow with $F(p, 0) = (p, u_0(p))$ ($p \in G/K$) for the $K$-invariant function $u_0$. 
Under this setting, we obtain the following theorems. 

\begin{mainthm}\label{thm:corHZZ}
Let $I = (-a, a)$ and $r \colon I \to \mathbb{R}$ be a positive smooth function. 
Assume that $r$ satisfies the following conditions:
\begin{description}
\item[(1)] $r(0) = 1$ and $r'(0) = 0$;
\item[(2)] $r'(z) > 0$ for all $z \in (0, a)$ and $r'(z) < 0$ for all $z \in (-a, 0)$;
\item[(3)] $r(z) r''(z) - r'(z)^2 \geq 0$.
\end{description} 
If $\widehat{u}_0$ satisfies 
\begin{equation*}
\max_{x \in\mathbb{R}} \abs{\widehat{u}_0(x)} < a_0, \quad \max_{x \in\mathbb{R}} \abs{\widehat{u}'_0(x)} < \sqrt{ \frac{n}{2(n+1) \log r(a_0)} }
\end{equation*}
for some $a_0 \in (0,a)$, then the mean curvature flow $F$ exists for all $t \in [0,\infty)$.
\end{mainthm}

\begin{mainthm}\label{thm:corF}
Let $I = (-\infty, a)$ and $r \colon I \to \mathbb{R}$ be a positive smooth function.
Assume that $r$ satisfies the following conditions:
\begin{description}
\item[(1)] $r'(z) > 0$ for all  $z \in (-\infty,a)$;
\item[(2)] $r(z) r''(z) - (1 + \alpha){r'(z)}^2 \geq 0$ for some positive constant $\alpha$ greater than one.
\end{description} 
If $\widehat{u}_0$ satisfies
\begin{equation*}
\min_{x \in \mathbb{R}} \widehat{u}_0(x) > a_1, \quad \max_{x \in \mathbb{R}} \abs{\widehat{u}'_0(x)} < \sqrt{\frac{n(\alpha - 1)}{n+1}} r(a_1)
\end{equation*}
for some constant $a_1$ less than $a$, then the mean curvature flow $F$ exists for all $t \in [0,\infty)$.
\end{mainthm}

\begin{figure}
\centering
\begin{tikzpicture}[domain=-5:5,samples=100,>=stealth]
\draw[->,dashed] (-5.5,0) -- (5.5,0) node[right] {$z$};

\draw (-6.5,-0.5) to [out=180,in=180] (-6.5,0.5);
\draw (-6.5,-0.5) to [out=0,in=0] (-6.5,0.5);
\draw (-6.5,-0.5) node[below] {$G/K$};

\draw(-5,1)to[out=-10,in=-170](5,1);
\draw(-5,-1)to[out=10,in=170](5,-1);

\draw[dotted] (-5,-1) to [out=180,in=180] (-5,1);
\draw[dotted] (-5,-1) to [out=0,in=0] (-5,1);
\draw (-5,-0.05)--(-5,0.05);
\draw (-5,0) node[above] {$-a$};

\draw (-3,-0.7) to [out=180,in=180] (-3,0.7);
\draw[dotted] (-3,-0.7) to [out=0,in=0] (-3,0.7);

\draw (0,-0.5) to [out=180,in=180] (0,0.5);
\draw[dotted] (0,-0.5) to [out=0,in=0] (0,0.5);
\draw[thick] (0, 0.05) -- (0,-0.05);
\draw (0,0) node[below] {$0$};

\draw (3,-0.7) to [out=180,in=180] (3,0.7);
\draw[dotted] (3,-0.7) to [out=0,in=0] (3,0.7);

\draw[dotted] (5,-1) to [out=180,in=180] (5,1);
\draw[dotted] (5,-1) to [out=0,in=0] (5,1);
\draw (5,-0.05)--(5,0.05);
\draw (5,0) node[above] {$a$};

\draw[red,thick] (2.5,-0.63) to [out=0,in=180] (3.5,0.75); 
\draw[dotted,red,thick] (2.5,-0.63) to [out=180,in=0] (3.5,0.75); 

\draw[red,thick] (0.8,-0.5) to [out=0,in=180] (1,0.51); 
\draw[dotted,red,thick] (0.8,-0.5) to [out=180,in=0] (1,0.51);
\draw (1.5,0.51) node[above] {$F_t(G/K)$};

\draw[->,dashed] (0.8,-0.7)--(0,-0.7);
\draw (0.4,-0.7) node[below]{$(t \to \infty)$};

\draw[red,thick] (0,-0.5) to [out=180,in=180] (0,0.5);
\draw[dotted,red,thick] (0,-0.5) to [out=0,in=0] (0,0.5);

\end{tikzpicture}
\caption{Theorem \ref{thm:corHZZ}}
\label{fig:thm}
\end{figure}

\begin{figure}
\centering
\begin{tikzpicture}[domain=-5:5,samples=100,>=stealth]
\draw[->,dashed] (-5.5,0) -- (5.5,0) node[right] {$z$};

\draw (-6.5,-0.5) to [out=180,in=180] (-6.5,0.5);
\draw (-6.5,-0.5) to [out=0,in=0] (-6.5,0.5);
\draw (-6.5,-0.5) node[below] {$G/K$};

\draw(-5,0.1)to[out=1,in=-170](5,0.7);
\draw(-5,-0.1)to[out=-1,in=170](5,-0.7);

\draw (-4,-0.1) to [out=180,in=180] (-4,0.1);
\draw[dotted] (-4,-0.1) to [out=0,in=0] (-4,0.1);

\draw (0,-0.17) to [out=180,in=180] (0,0.17);
\draw[dotted] (0,-0.17) to [out=0,in=0] (0,0.17);

\draw (4,-0.53) to [out=180,in=180] (4,0.53);
\draw[dotted] (4,-0.53) to [out=0,in=0] (4,0.53);

\draw[dotted] (5,-0.7) to [out=180,in=180] (5,0.7);
\draw[dotted] (5,-0.7) to [out=0,in=0] (5,0.7);
\draw (5,-0.05)--(5,0.05);
\draw (5,0) node[above] {$a$};

\draw[red,thick] (2.5,-0.35) to [out=0,in=180] (3.5,0.45); 
\draw[dotted,red,thick] (2.5,-0.35) to [out=180,in=0] (3.5,0.45); 

\draw[red,thick] (0.3,-0.17) to [out=0,in=180] (0.5,0.19); 
\draw[dotted,red,thick] (0.3,-0.17) to [out=180,in=0] (0.5,0.19); 

\draw[red,thick] (-3.7,-0.1) to [out=180,in=180] (-3.65,0.1);
\draw[dotted,red,thick] (-3.7,-0.1) to [out=0,in=0] (-3.65,0.1);
\draw (-3.3,0.1) node[above] {$F_t(G/K)$};
\draw[->,dashed] (-3.7,-0.3)--(-5,-0.3);
\draw (-4.35,-0.3) node[below]{$(t \to \infty)$};

\end{tikzpicture}  

\caption{Theorem \ref{thm:corF}}
\label{fig:thm}
\end{figure}

Theorem \ref{thm:corHZZ} corresponds to Theorem 1.2 of \cite{HZZ} and Theorem \ref{thm:corF} corresponds to Theorem 1.2 of \cite{F}.

This paper is organized as follows. In Section \ref{sec:warped}, we give an explicit expression of the mean curvature of graph hypersurface in the general warped product $_r M \times I$, and in Sections \ref{sec:mean  curvature} and \ref{sec:invariant mcf}, we investigate the $K$-invariant graph hypersurface and the $K$-invariant graphical mean curvature flow in ${}_r G/K \times I$ by using the explicit expression of the mean curvature in Section \ref{sec:warped} repectively, where $G/K$ is a rank one symmetric space of compact type. Theorems \ref{thm:corHZZ} and \ref{thm:corF} are proven in Section \ref{sec:proof}.

\section{Explicit formula of mean curvature}\label{sec:warped}

We recall the definition of warped products. 
Let $(M, g_M)$ be an $n$-dimensional closed Riemannian manifold, and $r(z)$ a positive function defined on an open interval $I$.
A warped product $(\overline{M}, \overline{g})$ of $M$ and $I$ with a warping function $r$ is defined by a product manifold $M \times I$ equipped with the metric 
\begin{align*}
\overline{g} = ({\pi_I}^*r)^2 {\pi_M}^* g_M + \pi_I^* (dz \otimes dz),
\end{align*}
where $\pi_M \colon \overline{M} \to M$ and $\pi_I \colon \overline{M} \to I$ are the natural projections onto $M$ and $I$, respectively.
Let $\{ e^M_i \}$ be a local orthonormal frame field on $(M, g_M)$ and we define
\begin{align*}
    (\overline{E}_i)_{(p,z)} \coloneqq \frac{1}{r(z)} (e^M_i)^{L}_{(p,z)}, \quad (\overline{\partial_z})_{(p,z)} \coloneqq (\partial_z)^L_{(p,z)}
\end{align*}
where $(\cdot)^{L}$ denotes the natural lift to $\overline{M}$. 
The Levi-Civita connections of $g_M$ and $\overline{g}$ are denoted by $\nabla^M$ and $\overline{\nabla}$, respectively. 
Then, we have the following formulas (see \cite{O'Neill}, for example).
\begin{lemma}\label{lem:covariant}
We have the following formulas:
\begin{align*}
    &\overline{\nabla}_{\overline{E}_i} \overline{E}_j = \frac{1}{r^2} \left( \nabla^{M}_{e^M_i} e^M_j \right)^L - \frac{r'}{r} \delta_{i j} \overline{\partial_z}, \quad \overline{\nabla}_{\overline{\partial_z}} \overline{E}_i = 0, \\
    &\overline{\nabla}_{\overline{E}_i} \overline{\partial_z} = \frac{r'}{r} \overline{E}_i, \quad \overline{\nabla}_{\overline{\partial_z}} \overline{\partial_z} = 0.
\end{align*}
\end{lemma}

Let $u \colon M \to I$ be a smooth function and define $f \colon M \to \overline{M}$ by $f(p) \coloneqq (p, u(p))$. 
In the remainder of this section, we calculate the mean curvature of $f(M)$ in the same way as \cite{BM}. 
We use the same local frame fields $\{ e^M_i \}$ and $\{ \overline{E}_1, \dots, \overline{E}_n, \overline{\partial_z} \}$. 
Set $u_i \coloneqq e^M_i (u)$ and $u_{i j} \coloneqq e^M_j ( e^M_i (u) )$. 
Then, we have
\begin{equation*}
d f(e^M_i) = (r \circ u) \overline{E}_i + u_i \overline{\partial_z},
\end{equation*}
so the induced metric $g \coloneqq f^* \overline{g}$ is locally given as follows:
\begin{equation*}
g_{i j} = g (e^M_i, e^M_j) = (r \circ u)^2 \delta_{i j} + u_i u_j.
\end{equation*}
The inverse matrix $(g^{i j})$ of $(g_{i j})$ is given by
\begin{equation*}
g^{i j} = \frac{1}{(r \circ u)^2} \left( \delta^{i j} - \frac{u_i u_j}{w^2} \right),
\end{equation*}
where $w^2 \coloneqq (r \circ u)^2 + \abs{\mathrm{grad}\, u}^2$.
The unit normal vector field $N$ is given by
\begin{equation*}
N = \frac{1}{w} \left( - \sum_{i = 1}^n u_i \overline{E}_i + (r \circ u)  \overline{\partial_z} \right).
\end{equation*} 
Using above formulas, we then calculate the second fundamental form $h = ( h_{i j} )$ with respect to $- N$. 
First, we have the following:
\begin{align*}
h_{i j} =&
\overline{g} \left( \overline{\nabla}^f_{e^M_i} N, d f(e^M_j) \right) \\
=& \frac{1}{w} \overline{g} \left( \overline{\nabla}^f_{e^M_i} \left( - \sum_{k = 1}^n u_k \overline{E}_k + (r \circ u)  \overline{\partial_z} \right), d f(e^M_j) \right) \\
=&  \frac{1}{w} \overline{g} \left( - \sum_{k = 1}^n \left( u_{k i} \overline{E}_k + u_k \overline{\nabla}_{d f(e^M_i)} \overline{E}_k \right), (r \circ u) \overline{E}_j + u_j \overline{\partial_z} \right) \\
&+ \frac{1}{w} \overline{g} \left(  (r' \circ u) u_i \overline{\partial_z} + (r \circ u) \overline{\nabla}_{d f(e^M_i)} \overline{\partial_z} , (r \circ u) \overline{E}_j + u_j \overline{\partial_z} \right) \\
=&- \frac{(r \circ u) u_{j i}}{w} - \sum_{k = 1}^n \frac{u_k}{w} \overline{g} \left( \overline{\nabla}_{d f(e^M_i)} \overline{E}_k, (r \circ u) \overline{E}_j + u_j \overline{\partial_z} \right) \\
&+ \frac{(r' \circ u) u_i u_j}{w} + \frac{r \circ u}{w} \overline{g} \left( \overline{\nabla}_{d f(e^M_i)} \overline{\partial_z} , (r \circ u) \overline{E}_j + u_j \overline{\partial_z} \right),
\end{align*}
where $\overline{\nabla}^f$ denotes the pull-back connection of $\overline{\nabla}$ by $f$.
By using Lemma \ref{lem:covariant}, we have
\begin{align*}
\overline{\nabla}_{d f(e^M_i)} \overline{E}_k 
=& (r \circ u) \overline{\nabla}_{\overline{E}_i} \overline{E}_k + u_i \overline{\nabla}_{\overline{\partial_z}} \overline{E}_k = \frac{1}{r \circ u} \left( \nabla^{M}_{e^M_i} e^M_k \right)^L - (r' \circ u) \delta_{i k} \overline{\partial_z}, \\
\overline{\nabla}_{d f(e^M_i)} \overline{\partial_z} 
=& (r \circ u) \overline{\nabla}_{\overline{E}_i} \overline{\partial_z} + u_i \overline{\nabla}_{\overline{\partial_z}} \overline{\partial_z} = (r' \circ u) \overline{E}_i.
\end{align*}
Therefore, we obtain
\begin{align*}
h_{i j}
=& - \frac{(r \circ u) u_{j i}}{w} - \sum_{k = 1}^n \frac{u_k}{w} \overline{g} \left( \left( \nabla^{M}_{e^M_i} e^M_k \right)^L,\overline{E}_j \right) + \frac{(r' \circ u) u_i u_j}{w} \\
&+ \frac{(r' \circ u) u_i u_j}{w} + \frac{(r \circ u)^2 (r' \circ u)}{w} \delta_{i j} \\ 
=& - \frac{(r \circ u) u_{j i}}{w} - \sum_{k = 1}^n \frac{(r \circ u) u_k}{w} g_M \left( \nabla^{M}_{e^M_i} e^M_k,e^M_j \right) + \frac{(r' \circ u) u_i u_j}{w}+ \frac{r' \circ u}{w} g_{i j} \\
=& - \frac{r \circ u}{w} \left( \nabla^M \nabla^M u \right) (e^M_i, e^M_j) + \frac{(r' \circ u) u_i u_j}{w}+ \frac{r' \circ u}{w} g_{i j},
\end{align*}
where we use $g_M \left( \nabla^{M}_{e^M_i} e^M_k,e^M_j \right) = - g_M \left(e^M_k,  \nabla^{M}_{e^M_i} e^M_j \right)$.
Thus, we have 
\begin{align*}
H 
=& \sum_{i, j = 1}^n g^{i j} h_{i j} \\
=& - \sum_{i, j = 1}^n \frac{1}{(r \circ u) w} \left( \delta^{i j} - \frac{u_i u_j}{w^2} \right) \left( \nabla^M \nabla^M u \right) (e^M_i, e^M_j) \\
&  + \sum_{i, j = 1}^n \frac{(r' \circ u) u_i u_j}{(r \circ u)^2 w} \left( \delta^{i j} - \frac{u_i u_j}{w^2} \right) + n \frac{r' \circ u}{w} \\
=& - \frac{1}{(r \circ u) w} \Delta^M u + \frac{1}{(r \circ u) w^3} \left( \nabla^M \nabla^M u \right) \left( \mathrm{grad}\, u, \mathrm{grad}\, u \right) \\
&+ \frac{r' \circ u}{(r \circ u)^2 w^3} \abs{\mathrm{grad}\, u}^2 \left( w^2 - \abs{\mathrm{grad}\, u}^2 \right) + n \frac{r' \circ u}{w} \\
=&- \frac{1}{(r \circ u) w} \Delta^M u + \frac{1}{(r \circ u) w^3} \left( \nabla^M \nabla^M u \right) \left( \mathrm{grad}\, u, \mathrm{grad}\, u \right) + \frac{r' \circ u}{w^3} \abs{\mathrm{grad}\, u}^2 + n \frac{r' \circ u}{w}.
\end{align*}
From the above argument, we obtain the following expression of the mean curvature.
\begin{lemma}\label{lem:mean curvature}
The mean curvature of the graph hypersurface of $u$ is given by
\begin{equation*}
H = - \frac{1}{(r \circ u) w} \Delta^M u + \frac{1}{2 (r \circ u) w^3} g_M \left( \mathrm{grad} \left( \abs{\mathrm{grad}\, u}^2 \right), \mathrm{grad}\, u \right) + \frac{r' \circ u}{w^3} \abs{\mathrm{grad}\, u}^2 + n \frac{r' \circ u}{w},
\end{equation*}
where $w^2 = (r \circ u)^2 + \abs{\mathrm{grad}\, u}^2$.
\end{lemma}
\begin{proof}
We have the following equality:
\begin{align*}
\left( \nabla^M \nabla^M u \right) \left(\mathrm{grad}\, u, \mathrm{grad}\, u \right) 
=& \left(\mathrm{grad}\, u\right) g_M \left( \mathrm{grad}\, u, \mathrm{grad}\, u \right) - g_M \left( \mathrm{grad}\, u, \nabla^M_{\mathrm{grad}\, u} \mathrm{grad}\, u\right) \\
=& g_M \left(\nabla^M_{\mathrm{grad}\, u} \mathrm{grad}\, u, \mathrm{grad}\, u \right) \\
=& \frac{1}{2} \left( \mathrm{grad}\, u \right) \abs{\mathrm{grad}\, u}^2 \\
=& \frac{1}{2} g_M \left( \mathrm{grad} \left( \abs{\mathrm{grad}\, u}^2 \right), \mathrm{grad}\, u \right).
\end{align*}
Using this equality, the lemma follows immediately.
\end{proof}

\section{Mean curvature of isotropy invariant graph}\label{sec:mean curvature}
In this section, we consider the case where $(M, g_M)$ is a Riemannian symmetric space $G/K$ of compact type.
Let $\mathfrak{g}$ and $\mathfrak{k}$ be the Lie algebras of $G$ and $K$, respectively,  and $\theta$ be the involution of $G$ with $(\mathrm{Fix} \, \theta)_0 \subset K \subset \mathrm{Fix} \, \theta$, where $\mathrm{Fix} \, \theta$ denotes the fixed point group of $\theta$ and $(\mathrm{Fix}\, \theta)_0$ is the identity component of $\mathrm{Fix} \, \theta$. 
Let the same symbol $\theta$ denote the involution of $\mathfrak{g}$ induced from $\theta$. 
Set $\mathfrak{p} \coloneqq \mathrm{Ker}(\theta + \mathrm{id}_{\mathfrak{g}})$. 
We have the decomposition $\mathfrak{g} = \mathfrak{k} \oplus \mathfrak{p}$, which is called the \textit{canonical decomposition} associated with the symmetric pair $(G, K)$. 
We take the $G$-invariant metric determined by $- B \vert_{\mathfrak{p} \times \mathfrak{p}}$ as a Riemannian metric on $G/K$, where $B$ is the Killing form of $\mathfrak{g}$, and we identify $\mathfrak{p}$ with $T_{e K}(G/K)$ through the linear isomorphism $\pi_{\ast}\vert_{\mathfrak{p}} \colon \mathfrak{p} \to T_{e K}(G/K)$, where $\pi$ is the natural projection of $G$ onto $G/K$.
The isotropy subgroup $K$ of $G$ at $e K$ acts on $M = G/K$ naturally, and is called an isotropy action on $G/K$.
Let $k$ be the rank of the symmetric space $G/K$. Let 
\begin{equation}\label{root decomposition}
    \mathfrak{p} = \mathfrak{a} \oplus \left( \oplus_{\lambda \in \Delta_+} \mathfrak{p}_{\lambda} \right)
\end{equation}
be the restricted root space decomposition of $\mathfrak{p}$ with respect to a maximal abelian subspace $\mathfrak{a}$ of $\mathfrak{p}$, where $\Delta_{+}$ is the positive root system under some lexicographical ordering of $\mathfrak{a}$.
Define a domain $C'$ of $\mathfrak{a}$ by
\begin{align*}
    C' = \{ v \in \mathfrak{a} \mid 0 < \lambda(v) < \pi \, (\lambda \in \Delta_+)  \}
\end{align*}
and set $\mathcal{U}' \coloneqq K \cdot \mathrm{exp}_{e K} ( C' ) = \bigcup_{p \in \mathrm{exp}_{e K} (C')} K \cdot p$, where we note that $\dim \mathfrak{a} = k$ and that $K \cdot p$ ($p \in \mathcal{U}'$) are principal orbits. 
It is known that $\mathcal{U}'$ is an open dense subset of $M$.
Then, we construct a local orthonormal frame field on $\mathcal{U}'$ that is suitable for our purpose.

\begin{lemma}\label{lem:frame on M}
Let $\{ e^0_i \}_{i = 1, \dots, m_0}$ and $\{ e^{\lambda}_i \}_{i = 1, \dots, m_{\lambda}}$ be orthonormal bases of $\mathfrak{a}$ and $\mathfrak{p}_{\lambda}$, respectively, where $m_0 = \dim \mathfrak{a}( = k)$ and $m_{\lambda} = \dim \mathfrak{p}_{\lambda}$.
Then, there is a local orthonormal frame field 
\begin{align*}
\mathcal{F}^M \coloneqq \bigcup_{\bullet \in \{0\} \cup \Delta_+} \left\{ (e^{\bullet}_1)^M, \dots, (e^{\bullet}_{m_{\bullet}})^M \right\}
\end{align*}
on $\mathcal{U}'$ consisting of $K$-equivariant vector fields satisfying two conditions: 
\begin{description}
\item{(I)}
$(e^0_i)^M_{e K} = e^0_i \in \mathfrak{a}$, $(e^{\lambda}_i)^M_{e K} = e^{\lambda}_i \in \mathfrak{p}_{\lambda}$, and 
\begin{align*}
    \nabla^M_{(e^0_i)^M} (e^0_j)^M = 0, \quad \left( \nabla^M_{(e^{\lambda}_i)^M} (e^0_j)^M \right)_{k \cdot p} = \frac{\lambda(e^0_j)}{\tan \left( \sum_{l=1}^k \lambda(e^0_l) x^l \right)} (e^{\lambda}_i)^M,
\end{align*}
where $p = \mathrm{exp}_{e K}(\sum_{l = 1}^k x^l e^0_l) \in \mathrm{exp}_{e K}(C')$ and $k \in K$. 
\item{(II)}
For any $K$-invariant smooth function $u \colon M \to \mathbb{R}$, that is, the function $u$ that satisfies $u(p) = u(k \cdot p)$ for all $p \in M$ and $k \in K$, we have
\begin{align*}
    (e^{\lambda}_i)^M (u) = 0
\end{align*}
for $i = 1, \dots, m_{\lambda}$ and $\lambda \in \Delta_+$.
\end{description}
\end{lemma}
\begin{proof}
Set $v = \sum_{i = 1}^k x^i e^0_i$ and $p = \mathrm{exp}_{e K}(v) $, and let $\gamma$ be a geodesic defined by $\gamma(s) \coloneqq \mathrm{exp}_{e K}(s v)$, and $s_q \colon M \to M$ be the involutive isometry having $q$ as the isolated fixed point.
We set $T_s$ as $T_s = s_{\gamma(s/2)} \circ s_{\gamma(0)}$, a transvection along $\gamma$.
Then, we define a $K$-equivariant orthonormal frame field by
\begin{align*}
    (e^{\bullet}_i)^M_{k \cdot \gamma(s)} \coloneqq d k_{\gamma(s)} (d T_s (e^{\bullet}_i))
\end{align*}
for $\bullet \in \{ 0 \} \cup \Delta_+$. 
This frame field is defined on $\mathcal{U}'$ since $\mathrm{exp}_{e K}(C')$ intersects only once with each orbit $K \cdot p$ of the isotropy action.
We also have $T^{\perp}_{q} (K \cdot q) = d T_s (\mathfrak{a})$, and $T_{q} (K \cdot q) = d T_s (\left( \oplus_{\lambda \in \Delta_+} \mathfrak{p}_{\lambda} \right))$ for $q = \gamma(s)$, hence for $q' = k \cdot q \in \mathcal{U}'$, we have
\begin{align*}
    (e^0_i)^M_{q'} \in T^{\perp}_{q'} (K \cdot q), \quad (e^{\lambda}_j)^M_{q'} \in T_{q'}(K \cdot q),
\end{align*}
for $i = 1, \dots, k$ and $j = 1, \dots, m_{\lambda}$.
We show that this frame field satisfies the conditions in the lemma. 
Since $\mathrm{exp}_{e K}(\mathfrak{a})$ is a flat totally geodesic submanifold in $G / K$, we have
\begin{equation*}
\nabla^M_{(e^0_i)^M} (e^0_j)^M = 0.
\end{equation*}
Let us consider the principal orbit $K \cdot p$. From Weingarten formula, we have
\begin{align*}
\left( \nabla^M_{(e^{\lambda}_i)^M} (e^0_j)^M \right)_{p} = -A_{(e^0_j)^M_p} (e^{\lambda}_i)^M_p + \nabla^{\perp}_{(e^{\lambda}_i)^M} (e^0_j)^M,
\end{align*} 
where $A$ is a shape operator of $K \cdot p$ and $\nabla^{\perp}$ is a normal connection. 
It is known that $\nabla^{\perp}_{(e^{\lambda}_i)^M} (e^0_j)^M$ vanishes because $(e^0_j)^M$ is a $K$-equivariant normal vector field on the principal orbit $K \cdot p$ (see, for example, Theorem 5.6.7(6) or Theorem 5.7.1(1) in \cite{PT}), and by Theorem 1 in \cite{V}, the shape operator is given by
\begin{equation*}
A_{(e^0_j)^M_p} (e^{\lambda}_i)^M_p = - \frac{\lambda(e^0_j)}{\tan (\sum_{l=1}^k \lambda(e^0_l) x^l)} (e^{\lambda}_i)^M.
\end{equation*}
Hence, we obtain the covariant derivatives at every point in $\mathcal{U}'$ since the isotropy action is isometric. 
Lastly, if the function $u \colon M \to \mathbb{R}$ is $K$-invariant, we have
\begin{equation*}
(e^{\lambda}_i)^M_{k \cdot p} (u) = d T_1 (e^{\lambda}_i) (u \circ k) =  d T_1 (e^{\lambda}_i) (u) = \frac{d}{d t} \biggl\vert_{t = 0} u(c(t) \cdot p) = \frac{d}{d t} \biggl\vert_{t = 0} u(p) =  0, 
\end{equation*}
where $c(t)$ is a smooth curve in $K$ with $c(0) = e$ and $\dot{c}(0) = d T_1 (e^{\lambda}_i) \in T_p (K \cdot p) $.
\end{proof}

From now on, we use the orthonormal frame field $\mathcal{F}^M$ on $\mathcal{U}'$ defined in the proof of Lemma \ref{lem:frame on M}, and we define the orthonormal frame field 
\begin{align*}
\overline{\mathcal{F}} \coloneqq \left(\bigcup_{\bullet \in \{0\} \cup \Delta_+} \{ \overline{E}^{\bullet}_1, \dots, \overline{E}^{\bullet}_{m_{\bullet}} \} \right) \cup \{ \overline{\partial_z}\}
\end{align*}
on $\mathcal{U}' \times I$ by 
\begin{equation*}
(\overline{E}^{\bullet}_i)_{(p,z)} \coloneqq \frac{1}{r(z)} ((e^{\bullet}_i)^M)^L_{(p,z)}, \quad (\overline{\partial_z})_{(p,z)} \coloneqq (\partial_z)^L_{(p,z)},
\end{equation*}
where $(\cdot)^L$ is the lift to $\mathcal{U}' \times I$. 
Let $u \colon M \to \mathbb{R}$ be a $K$-invariant smooth function, and set $u_i \coloneqq (e^0_i)^M (u)$ and $u_{i j} \coloneqq (e^0_j)^M (e^0_i)^M (u)$. 
From Lemma \ref{lem:frame on M}, the function $u$ satisfies $(e^{\lambda}_i)^M (u) = 0$ for $i = 1, \dots, m_{\lambda}$ and $\lambda \in \Delta_+$. 
Note that $u_i$ is also $K$-invariant since 
\begin{equation*}
    u_i (k \cdot p) = (e^0_i)^M_{k \cdot p} (u) = (e^0_i)^M_p (u \circ k) = (e^0_i)^M_p (u) = u_i(p),
\end{equation*}
and this implies $(e^{\lambda}_j)^M (u_i) = 0$. 
Then, using the frame fields $\mathcal{F}^M$ and $\overline{\mathcal{F}}$, the following lemma holds.
\begin{lemma}\label{lem:laplacian}
Let $u \colon M \to I$ be a $K$-invariant smooth function. We have the following formulas:
\begin{align*}
\abs{\mathrm{grad}\, u}^2 
&= \sum_{i = 1}^k (u_i)^2, \quad
g_M \left(\mathrm{grad} \left( \abs{\mathrm{grad}\,  u}^2 \right), \mathrm{grad}\, u \right)= 2 \sum_{i, j = 1}^k u_{i j} u_i u_j, \\
\left( \Delta^M u \right) (k \cdot p)
&= \sum_{i = 1}^k u_{i i}(p) + \sum_{\lambda \in \Delta_+} \sum_{j = 1}^k  \frac{m_{\lambda} \lambda(e^0_j) u_j(p)}{\tan \left( \sum_{l=1}^k \lambda(e^0_l) x^l \right)},
\end{align*}
where $p = \mathrm{exp}_{e K} (\sum_{l = 1}^k x^l e^0_l) \in \mathrm{exp}_{e K}(C')$.
\end{lemma}
\begin{proof}
Since $u$ is $K$-invariant, we have $(e^{\lambda}_i)^M (u) = 0$. 
Hence, we can derive the first two formulas.
Also we have
\begin{align*}
\Delta^M u 
=&  \sum_{i = 1}^k \left( \left( (e^0_i)^M \circ (e^0_i)^M \right) - \nabla^M_{(e^0_i)^M} (e^0_i)^M \right) u + \sum_{\lambda \in \Delta_+} \sum_{i = 1}^{m_{\lambda}} \left( \left( (e^{\lambda}_i)^M \circ (e^{\lambda}_i)^M \right) - \nabla^M_{(e^{\lambda}_i)^M} (e^{\lambda}_i)^M \right) u \\
=& \sum_{i = 1}^k u_{i i} - \sum_{\lambda \in \Delta_+} \sum_{i = 1}^{m_{\lambda}} \left( \nabla^M_{(e^{\lambda}_i)^M} (e^{\lambda}_i)^M \right) u
\end{align*}
from Lemma \ref{lem:frame on M}. 
Since $u$ is $K$-invariant, we also have 
\begin{align*}
\left( \nabla^M_{(e^{\lambda}_i)^M } (e^{\lambda}_i)^M \right)_{k \cdot p} u 
=& \sum_{j = 1}^k g_M \left( \left( \nabla^M_{(e^{\lambda}_i)^M} (e^{\lambda}_i)^M \right)_{k \cdot p}, (e^0_j)^M_{k \cdot p} \right) (e^0_j)^M_{k \cdot p} u \\
=& - \sum_{j = 1}^k g_M \left((e^{\lambda}_i)^M_{k \cdot p}, \left( \nabla^M_{(e^{\lambda}_i)^M} (e^0_j)^M \right)_{k \cdot p} \right) u_j(k \cdot p) \\
=& - \sum_{j = 1}^k \frac{\lambda(e^0_j) u_j(p)}{\tan \left( \sum_{l=1}^k \lambda(e^0_l) x^l \right)},
\end{align*}
where $p = \mathrm{exp}_{e K} (\sum_{l = 1}^k x^l e^0_l) \in \mathrm{exp}_{e K}(C')$. 
Thus, we obtain the last formula.
\end{proof}
We then consider a smooth map $f \colon M \to \overline{M}$ defined by a graph of $K$-invariant function $u$, that is, $f(p) = (p,u(p))$. 
From Lemma \ref{lem:mean curvature} and Lemma \ref{lem:laplacian}, we obtain the following lemma.

\begin{lemma}\label{lem:mean curvature of sym}
Let $p = \mathrm{exp}_{e K} (\sum_{l = 1}^k x^l e^0_l) \in \mathrm{exp}_{e K}(C')$.  
Then the mean curvature $H$ of $f(M)$ with respect to $-N$ is given by
\begin{align*}
H(k \cdot p)
=& - \frac{1}{r(u(p)) w(p)} \sum_{i, j = 1}^k \left( \delta_{i j} - \frac{u_i(p) u_j(p)}{w(p)^2} \right) u_{i j}(p) + \frac{r'(u(p))}{w(p)^3} \sum_{i = 1}^k u_i(p)^2 \\
&- \frac{1}{r(u(p)) w(p)}  \sum_{\lambda \in \Delta_+} \sum_{j = 1}^k \frac{m_{\lambda} \lambda(e^0_j) u_j(p)}{\tan \left( \sum_{l = 1}^k  \lambda(e^0_l) x^l \right)} + n \frac{r'(u(p))}{w(p)},
\end{align*}
where $w^2(p) = r(u(p))^2 + \sum_{l = 1}^k u_l(p)^2$.
\end{lemma}
In the case where $M = G/K$ is a rank one symmetric space, that is, the sphere $S^{n}$, the complex projective space $\mathbb{CP}^{n/2}$, the quaternionic projective space $\mathbb{QP}^{n/4}$ or the Cayley plane $\mathbb{OP}^2$ (then $n = 16$), the restricted root space decomposition \eqref{root decomposition} is as follows: 
\begin{align*}
    \mathfrak{p} = \mathfrak{a} \oplus \mathfrak{p}_{\lambda} \oplus \mathfrak{p}_{2 \lambda},
\end{align*}
where $m_{2 \lambda} = \dim \mathfrak{p}_{2 \lambda}$ is given by
\begin{equation*}
m_{2 \lambda} = 
\begin{cases}
0 & (\text{when } M = S^n)\\
1 & (\text{when } M = \mathbb{CP}^{n/2})\\
3 & (\text{when } M = \mathbb{QP}^{n/4})\\
7 & (\text{when } M = \mathbb{OP}^2).
\end{cases}
\end{equation*}
Hence, $m_{\lambda} = \dim \mathfrak{p}_{\lambda}$ is given by
\begin{equation*}
m_{\lambda} = 
\begin{cases}
n - 1 & (\text{when } M = S^n)\\
n - 2 & (\text{when } M = \mathbb{CP}^{n/2})\\
n - 4 & (\text{when } M = \mathbb{QP}^{n/4})\\
8 & (\text{when } M = \mathbb{OP}^2).
\end{cases}
\end{equation*}
Then we have the following expression of the mean curvature.
\begin{cor}
Let $M = G/K$ be a rank one symmetric space of compact type and $p = \mathrm{exp}_{e K} (x e^0_1) \in \mathrm{exp}_{e K}(C')$. Then, the mean curvature $H$ of $F(M)$ with respect to $-N$ is given by
\begin{align*}
H(k \cdot p)
=& - \frac{r(u(p))}{w(p)^3} u''(p) + \frac{r'(u(p))}{w(p)^3}  u'(p)^2 \\
&- \frac{1}{r(u(p)) w(p)}  \left( \frac{m_{\lambda} \lambda(e^0_1) u'(p)}{\tan \left( \lambda(e^0_1) x \right)} + \frac{2 m_{2 \lambda} \lambda(e^0_1) u'(p)}{\tan \left( 2 \lambda(e^0_1) x \right)} \right) + n \frac{r'(u(p))}{w(p)},
\end{align*}
where $u' \coloneqq (e^0_1)^M (u)$ and $w^2(p) = r(u(p))^2 + u'(p)^2$.
\end{cor}

\section{Isotropy invariant graphical mean curvature flow}\label{sec:invariant mcf}
In this section, we study the mean curvature flow which is invariant under the $K$-action on $\overline{M}$. 
First, we show the following fact under the general setting.
\begin{lemma}\label{lem:K-equivMCF}
Let $(M, g_M)$ be a closed Riemannian manifold, $\mu \colon K \times M \to M$ an isometric action of a compact Lie group $K$ on $(M, g_M)$, and $\overline{\mu} \colon K \times \overline{M} \to \overline{M}$ the action on the warped product $\overline{M} = {_r M} \times I$ defined by $\overline{\mu}(k, (p,z)) \coloneqq (\mu(k, p), z)$.
\begin{description}
\item{(i)} This action $\overline{\mu} \colon K \times \overline{M} \to \overline{M}$ is isometric.
\item{(ii)} If $F \colon M \times [0,T) \to \overline{M}$ is a mean curvature flow such that $F_0 \coloneqq F(\cdot, 0) \ \colon M \to (\overline{M}, \overline{g})$ is $K$-equivariant (with respect to $\mu$ and $\overline{\mu}$), then $F_t \coloneqq F(\cdot, t)$ is also $K$-equivariant for all $t \in (0,T)$.
\end{description}
\end{lemma}
\begin{proof}
First, we show the statement (i).
Set $\mu^k \coloneqq \mu(k, \cdot)$ and  $\overline{\mu}^k \coloneqq \overline{\mu}(k, \cdot)$ for $k \in K$. We have
\begin{align*}
{\overline{\mu}^k}^* \overline{g} 
&= {\overline{\mu}^k}^* ( ({\pi_I}^* r)^2 ({\pi_M}^* g_M) + \pi_I^* (dz \otimes dz)) \\
&=  ({\pi_I}^* r)^2 ( {\pi_M}^*{\mu^k}^* g_M) + \pi_I^* (dz \otimes dz) \\
&=  ({\pi_I}^* r)^2 ({\pi_M}^* g_M) + \pi_I^* (dz \otimes dz) \\
&= \overline{g},
\end{align*}
hence the action $\overline{\mu}$ is isometric. 
Next we show the statemet (ii). 
Let $F$ be the mean curvature flow as in the statement (ii).
Then define $\widehat{F} \colon M \times [0,T) \to \overline{M}$ by $\widehat{F}  \coloneqq \overline{\mu}^k \circ F$. The map $\widehat{F}$ is also the mean curvature flow, in fact, $\widehat{F}$ satisfies 
\begin{equation*}
\frac{\partial \widehat{F}}{\partial t} 
= d \overline{\mu}^k \left( \frac{\partial F}{\partial t} \right) 
= d \overline{\mu}^k (\bm{H}_t) 
= \widehat{\bm{H}}_t,
\end{equation*}
where $\bm{H}_t$ and $\widehat{\bm{H}}_t$ are the mean curvature vectors of $F_t(M)$ and $\widehat{F}_t(M)$, respectively. 
On the other hand, $\widetilde{F} \coloneqq F(\mu^k(\cdot), \cdot)$ is also the mean curvature flow. 
The initial conditions $\widehat{F}_0$ and $\widetilde{F}_0$ coincide, hence we have $\widehat{F} = \widetilde{F}$ by the unique existence of the solution of the initial value problem of the mean curvature flow equation. 
This implies that $F_t$ is $K$-equivariant for all $t \in (0,T)$.
\end{proof}

Let $\psi_t \coloneqq \pi_M \circ F_t \colon M \to M$. 
Assume the initial condition is given by $F_0 (p) = (p, u_0(p))$ ($p \in M$), where $u_0 \colon M \to I$ is a smooth function. This guarantees that $\psi_0 = \pi_M \circ F_0 = \mathrm{id}_M$, and there exists a positive constant $\varepsilon_1$ such that $\psi_t$ is a diffeomorphism for $t \in [0, \varepsilon_1)$. 
We define 
\begin{equation*}
\widehat{T} \coloneqq \sup \left\{ t \in [0,T) \,\vert\, \psi_t \text{ is a diffeomorphism} \right\} (> 0).
\end{equation*}
Then, $\psi_t$ is a diffeomorphism for all $t \in [0,\widehat{T})$, and set $\phi_t \coloneqq (\psi_t)^{-1}$ and $\phi(\cdot, t) \coloneqq \phi_t (\cdot)$. 
The map $\phi \colon M \times [0, \widehat{T}) \to M$ is smooth. 
By the construction of $\phi_t$, it satisfies $\pi_M \circ F(\phi_t(p), t) = p$ ($p \in M$), and hence
\begin{equation*}
F(\phi_t(p), t) = (p, \pi_I \circ F(\phi_t(p), t))
\end{equation*}
holds. 
Define a function $u$ by $u(p, t) \coloneqq \pi_I \circ F(\phi_t(p), t)$  ($p \in M$). This function is the desired function. 
\begin{lemma}\label{lem:K-equivfunct}
Let $u_0 \colon M \to I$ be a smooth $K$-invariant function and set $F_0(p) \coloneqq (p, u_0(p))$  ($p \in M$). 
Assume that $F \colon M \times [0,T) \to \overline{M}$ is a mean curvature flow with $F(\cdot, 0) = F_0$ and $u \colon M \times [0,\widehat{T}) \to I$ is a function obtained from $F$ as above. 
Then, $u_t \coloneqq u(\cdot, t) \colon M \to I$ is $K$-invariant for all $t \in [0,\widehat{T})$.
\end{lemma}
\begin{proof}
We use the same notation as in Lemma \ref{lem:K-equivMCF} and the preceding argument. 
From Lemma \ref{lem:K-equivMCF}, we observe that the mean curvature flow $F$ is $K$-equivariant.
First, we show that the diffeomorphisms $\phi_t$ are $K$-equivariant for all $t \in [0,\widehat{T})$. 
Using the $K$-equivariance of $\pi_M$ and $F_t$, we have $d \mu^k \circ d (\pi_M \circ F_t) = d (\pi_M \circ F_t) \circ d \mu^k$. Hence, 
\begin{align*}
\mu^k \circ \psi_t = \mu^k \circ \pi_M \circ F_t = \pi_M \circ \overline{\mu}^k \circ F_t = \pi_M \circ F_t \circ \mu^k = \psi_t \circ \mu^k
\end{align*}
Thus, we have
\begin{equation*}
\mu^k \circ \phi_t = \phi_t \circ \mu^k.
\end{equation*}
Then, from the $K$-equivariance of $\phi_t$, we obtain the $K$-invariance of $u$ as follows:
\begin{align*}
u_t \circ \mu^k 
= \pi_I \circ F_t \circ \phi_t \circ \mu^k
= \pi_I \circ F_t \circ \mu^k \circ \phi_t
= \pi_I \circ \overline{\mu}^k \circ F_t \circ \phi_t 
= \pi_I \circ F_t \circ \phi_t
= u_t.
\end{align*}
\end{proof}
Now, by applying Lemma \ref{lem:K-equivfunct} to the case of the isotropy action $K \curvearrowright G/K$ of a symmetric space $M = G/K$ of compact type, we will derive the flow equation for $u_t$ corresponding to the mean curvature flow equation.

Let $u_0 \colon M \to I$ be a $K$-invariant smooth function and $F \colon M \times [0,T) \to \overline{M}$ a mean curvature flow with $F_0(p) = (p, u_0(p))$ ($p \in M$). Then, from Lemma \ref{lem:K-equivfunct}, we obtain the $K$-invariant smooth functions $u_t \colon M \to I$ ($t \in [0,\widehat{T})$). 
Similarly, we consider the graph of $u_t$, denoted by $\widehat{F}(p,t) \coloneqq (p, u_t(p))$ ($p \in M$), and we derive the flow equation of $u \colon M \times [0,T) \to I$ defined by $u(p,t) \coloneqq u_t(p)$ ($(p,t) \in M \times I$) from the following equation
\begin{equation*}
    \overline{g}\left( \frac{\partial \widehat{F}}{\partial t}, -N_t \right) = H_t.
\end{equation*}
The left-hand side is given by
\begin{align*}
\overline{g}\left( \frac{\partial \widehat{F}}{\partial t}, -N_t \right) 
= \overline{g} \left( \frac{\partial u}{\partial t} \overline{\partial_z},   -\sum_{j = 1}^k \frac{u_j}{w} \overline{E}^0_j + \frac{r \circ u}{w} \overline{\partial_{z}}   \right)
=- \frac{r \circ u}{w} \frac{\partial u}{\partial t}.
\end{align*}
Hence, from Lemma \ref{lem:mean curvature of sym}, we have
\begin{equation}\label{eq:flow equation}
\begin{aligned}
\frac{\partial u}{\partial t}(p,t) 
&=\frac{1}{r(u(p,t))^2} \sum_{i, j = 1}^k \left( \delta_{i j} - \frac{u_i(p,t) u_j(p,t)}{w(p,t)^2}\right) u_{i j}(p,t) - \frac{r'(u(p,t))}{r(u(p,t))} \sum_{i = 1}^k  \frac{u_i(p,t)^2}{w(p,t)^2} \\
&+ \frac{1}{r(u(p,t))^2}  \sum_{\lambda \in \Delta_+} \sum_{j = 1}^k \frac{m_{\lambda} \lambda(e^0_j) u_j(p,t)}{\tan \left( \sum_{l = 1}^k  \lambda(e^0_l) x^l \right)} - n \frac{r'(u(p,t))}{r(u(p,t))}
\end{aligned}
\end{equation}
on the open subset $\mathcal{U'} \times I \subset \overline{M}$, where $w(p,t)^2 \coloneqq r(u(p,t))^2 + \sum_{l=1}^k u_l(p,t)^2$. 
Define a function $\widehat{u} \colon \mathbb{R}^k \times [0,\widehat{T}) \to I$ by 
\begin{equation}\label{def:u-hat}
\widehat{u}(x^1, \dots, x^k, t) \coloneqq u \left( \mathrm{exp}_{e K}\left(\sum_{l = 1}^k x^l e^0_l \right), t \right)
\end{equation}
for $\sum_{l = 1}^k x^l e^0_l \in \mathfrak{a}$. 
In particular, for $\sum_{l = 1}^k x^l e^0_l \in C'$, we have
\begin{align*}
\widehat{u}_i(x,t) \coloneqq \frac{\partial \widehat{u}}{\partial x^i}(x, t) &=  \frac{\partial}{\partial x^i}u \left( \mathrm{exp}_{e K} \left(\sum_{l = 1}^k x^l e^0_l \right), t \right) = ((e^0_i)^M (u))(p, t) = u_i (p,t) \\
\widehat{u}_{ j i} (x,t) \coloneqq \frac{\partial^2 \widehat{u}}{\partial x^i \partial x^j}(x,t) &= \frac{\partial}{\partial x^i} u_j \left( \mathrm{exp}_{e K} \left(\sum_{l = 1}^k x^l e^0_l \right), t \right) = u_{ j i} (p, t),
\end{align*}
where $x = (x^1, \dots, x^k)$ and $p = \mathrm{exp}_{e K} \left(\sum_{l = 1}^k x^l e^0_l \right)$.
Therefore, we obtain 
\begin{align*}
\frac{\partial \widehat{u}}{\partial t}(x,t)
&= \frac{1}{r(\widehat{u}(x,t))^2} \sum_{i, j = 1}^k \left(\delta_{i j} - \frac{\widehat{u}_i(x,t) \widehat{u}_j(x,t)  }{\widehat{w}(x,t)^2} \right) \widehat{u}_{i j}(x,t) 
- \frac{r'(\widehat{u}(x,t))}{r(\widehat{u}(x,t)} \sum_{i = 1}^k \frac{\widehat{u}_i(x,t)^2}{\widehat{w}(x,t)^2}  \\
&+ \frac{1}{r(\widehat{u}(x,t))^2} \sum_{\lambda \in \Delta_+} \sum_{i = 1}^k \frac{m_{\lambda} \lambda(e^0_i)  \widehat{u}_i(x,t) }{\tan \left( \sum_{l = 1}^k  \lambda(e^0_l) x^l \right)} -n \frac{r'(\widehat{u}(x,t))}{r(\widehat{u}(x,t))},
\end{align*}
where $\widehat{w}(x,t)^2 \coloneqq r(\widehat{u}(x,t))^2 + \sum_{l=1}^k \widehat{u}_l(x,t)^2$.
 
Let $\widehat{C}$ denote a domain defined by
\begin{align}\label{def:C-hat}
\widehat{C} \coloneqq \left\{ (x^1, \dots, x^k) \in \mathbb{R}^k \biggm\vert   \sum_{l = 1}^k x^l e^0_l \in C' \right\}.
\end{align}
The closure of $\widehat{C}$ is a solid $k$-polytope and the boundary $\partial \widehat{C}$ is consisted of $(k+1)$ pieces of faces of the $k$-polytope. 
Let $\widehat{W}$ be the group generated by the reflections with respect to $(k-1)$-dimensional hyperplanes containing faces of the polytope. 
Each orbit of the isotropy action intersects with $\mathrm{exp}_{e K}(\overline{C'})$ only once, where $\overline{C'}$ denotes the closure of $C'$. For any $p \in M$, the set 
\begin{equation*}
\left\{ x \in \mathbb{R}^k \,\biggm\vert\,  \mathrm{exp}_{e K}\left(\sum_{l = 1}^k x^l e^0_l \right) \in \mathrm{exp}_{e K} (\mathfrak{a}) \cap (K \cdot p)  \right\}
\end{equation*}
is $\widehat{W}$-invariant. 
Since the function $u$ is $K$-invariant, the function $\widehat{u}(x, t) = u\left( \mathrm{exp}_{e K}\left(\sum_{l = 1}^k x^l e^0_l \right), t \right)$ defined on $\mathbb{R}^k \times [0,\widehat{T})$ is $\widehat{W}$-invariant. 
This together with \eqref{eq:flow equation} implies the following fact.
\begin{lemma}\label{lem:rank k}
Under the same hypothesis of Lemma \ref{lem:K-equivfunct}, let $\widehat{u} \colon \mathbb{R}^k \times [0, \widehat{T}) \to I$ be a smooth function defined by (\ref{def:u-hat}). 
Then, $\widehat{u}$ is a $\widehat{W}$-invariant function. 
Furthermore, $\widehat{u}$ satisfies the following:
\begin{equation}\label{eq:reduced}
\begin{aligned}
\frac{\partial \widehat{u}}{\partial t}(x,t)
&= \frac{1}{r(\widehat{u}(x,t))^2} \sum_{i, j = 1}^k \left(\delta_{i j} - \frac{\widehat{u}_i(x,t) \widehat{u}_j(x,t)  }{\widehat{w}(x,t)^2} \right) \widehat{u}_{i j}(x,t) 
- \frac{r'(\widehat{u}(x,t))}{r(\widehat{u}(x,t)} \sum_{i = 1}^k \frac{\widehat{u}_i(x,t)^2}{\widehat{w}(x,t)^2}  \\
&+ \frac{1}{r(\widehat{u}(x,t))^2} \sum_{\lambda \in \Delta_+} \sum_{i = 1}^k \frac{m_{\lambda} \lambda(e^0_i)  \widehat{u}_i(x,t) }{\tan \left( \sum_{l = 1}^k  \lambda(e^0_l) x^l \right)} -n \frac{r'(\widehat{u}(x,t))}{r(\widehat{u}(x,t))}
\end{aligned}
\end{equation}
for $(x,t) \in ( \widehat{W} \cdot \widehat{C}) \times [0,\widehat{T})$, where $\widehat{w}(x,t)^2 \coloneqq r(\widehat{u}(x,t))^2 + \sum_{l=1}^k \widehat{u}_l(x,t)^2$. 
\end{lemma}
In particular, we consider the case where $M = G/K$ is of rank one. 
Then we have
\begin{align*}
    \mathfrak{p} = \mathfrak{a} \oplus \mathfrak{p}_{\lambda} \oplus \mathfrak{p}_{2 \lambda},
\end{align*}
and
\begin{align*}
    C' = \{ v \in \mathfrak{a} \mid 0 < \lambda(v) < \pi, 0 < 2 \lambda(v) < \pi  \} = \{ v \in \mathfrak{a} \mid 0 < \lambda(v) < \pi/2  \}.
\end{align*}
Hence, the domain $\widehat{C}$ defined by (\ref{def:C-hat}) is given by
\begin{align*}
\widehat{C} = (0,\pi/2\lambda(e^0_1)),
\end{align*}
where $e^0_1$ is the unit vector in $\mathfrak{a}$ with $\lambda(e^0_1) > 0$.
\begin{cor}\label{cor:rank one}
Under the hypothesis of Lemma \ref{lem:rank k}, assume that $G/K$ is of rank one. 
Then we have
\begin{equation}\label{eq:rank one}
\begin{aligned}
\frac{\partial \widehat{u}}{\partial t}(x,t) 
&= \frac{ \widehat{u}''(x,t) }{r(\widehat{u}(x,t))^2 + \widehat{u}'(x,t)^2} 
- \frac{r'(\widehat{u}(x,t))}{r(\widehat{u}(x,t))}\frac{ \widehat{u}'(x,t)^2}{r(\widehat{u}(x,t))^2 + \widehat{u}'(x,t)^2}  \\
&+ \frac{\widehat{u}'(x,t)}{r(\widehat{u}(x,t))^2} \biggl( \frac{\lambda(e^0_1) m_{\lambda} }{\tan( \lambda(e^0_1) x)}  + \frac{ 2 \lambda(e^0_1) m_{2 \lambda} }{\tan(2 \lambda(e^0_1) x)} \biggr)
-n \frac{r'(\widehat{u}(x,t))}{r(\widehat{u}(x,t))}
\end{aligned}
\end{equation}
for $(x,t) \in (0,\pi/2\lambda(e^0_1)) \times [0, \widehat{T})$. 
By the $\widehat{W}$-invariance of $\widehat{u}$, $\widehat{u}$ satisfies the following boundary conditions:
\begin{equation*}
    \widehat{u}'(0,t) = \widehat{u}'(\pi/2 \lambda(e^0_1), t) = 0.
\end{equation*}
\end{cor}

\section{Proof of Theorems \ref{thm:corHZZ} and \ref{thm:corF}}\label{sec:proof}
In this section, we prove Theorems \ref{thm:corHZZ} and \ref{thm:corF} stated in Introduction. 
We shall use the notations in Section \ref{sec:invariant mcf}.
From the equation \eqref{eq:rank one}, the derivative $\widehat{u}'$ satisfies
\begin{equation*}
\begin{aligned}
\frac{\partial \widehat{u}'}{\partial t}(x,t) 
=& \frac{\widehat{u}'''(x,t)}{r(\widehat{u}(x,t))^2 + \widehat{u}'(x,t)^2} 
+ b(x, t) \widehat{u}''(x,t) -n \left( \frac{r'}{r} \right)' \hspace{-1.5mm} (\widehat{u}(x,t)) \widehat{u}'(x,t) \, \\
&- \left( \frac{r'}{r} \right)' \hspace{-1.5mm} (\widehat{u}(x,t)) \, \frac{\widehat{u}'(x,t)^3 }{ r(\widehat{u}(x,t))^2 + \widehat{u}'(x,t)^2} 
+ \frac{2 r'(\widehat{u}(x,t))^2 \widehat{u}'(x,t)^3 }{ ( r(\widehat{u}(x,t))^2 + \widehat{u}'(x,t)^2)^2} \\
&-\frac{\lambda(e^0_1) m_{\lambda} \widehat{u}'(x,t) }{r(\widehat{u}(x,t))^2}  \left( \frac{2 \widehat{u}'(x,t)}{\tan (\lambda(e^0_1) x)} \frac{r'(\widehat{u}(x,t))}{r(\widehat{u}(x,t))} + \frac{\lambda(e^0_1)}{\sin^2 (\lambda(e^0_1) x)} \right) \\
&-\frac{4 \lambda(e^0_1) m_{2 \lambda} \widehat{u}'(x,t) }{r(\widehat{u}(x,t))^2}  \left( \frac{\widehat{u}'(x,t)}{\tan (2 \lambda(e^0_1) x)} \frac{r'(\widehat{u}(x,t))}{r(\widehat{u}(x,t))} + \frac{\lambda(e^0_1)}{\sin^2 ( 2 \lambda(e^0_1) x)} \right),
\end{aligned}
\end{equation*}
where $b(x,t)$ is a certain smooth function. 
Hence, we obtain the following lemma.
\begin{lemma}\label{lem:mu ineq}
Let $\mu(x,t) \coloneqq \widehat{u}'(x,t)^2$, then the following inequality holds:
\begin{align*}
\frac{\partial \mu}{\partial t}(x,t) 
\leq& \frac{\mu''(x,t)}{r(\widehat{u}(x,t))^2  + \mu(x,t)} + \beta(x,t) \mu'(x,t)  - 2 n \left( \frac{r'}{r} \right)' \hspace{-1.5mm} (\widehat{u}(x,t)) \mu(x,t) \\
&- \left( \frac{r'}{r} \right)' \hspace{-1.5mm} (\widehat{u}(x,t)) \, \frac{ 2 \mu(x,t)^2 }{ r(\widehat{u}(x,t))^2 + \mu(x,t)} 
+  \frac{2 (n+1) \mu(x,t)^2}{r(\widehat{u}(x,t))^2} \frac{r'(\widehat{u}(x,t))^2}{r(\widehat{u}(x,t))^2},
\end{align*}
where $\beta(x,t)$ is a certain smooth function. 
\end{lemma}
\begin{proof}
The evolution equation for $\mu(x,t) \coloneqq \widehat{u}'(x,t)^2$ is given by 
\begin{equation}\label{eq:mu}
\begin{aligned}
\frac{\partial \mu}{\partial t}(x,t) 
=& \frac{\mu''(x,t)}{r(\widehat{u}(x,t))^2  + \mu(x,t)} + \beta(x,t) \mu'(x,t) - 2 n \left( \frac{r'}{r} \right)' \hspace{-1.5mm} (\widehat{u}(x,t)) \mu(x,t) \\
&- \left( \frac{r'}{r} \right)' \hspace{-1.5mm} (\widehat{u}(x,t)) \, \frac{ 2 \mu(x,t)^2 }{ r(\widehat{u}(x,t))^2 + \mu(x,t)} 
+ \frac{4 r'(\widehat{u}(x,t))^2 \mu(x,t)^2 }{ ( r(\widehat{u}(x,t))^2 + \mu(x,t))^2} \\
&-\frac{2 \lambda(e^0_1) m_{\lambda}\mu(x,t) }{r(\widehat{u}(x,t))^2}  \left( \frac{2 \widehat{u}'(x,t)}{\tan (\lambda(e^0_1) x)} \frac{r'(\widehat{u}(x,t))}{r(\widehat{u}(x,t))} + \frac{\lambda(e^0_1)}{\sin^2 (\lambda(e^0_1) x)} \right) \\
&-\frac{8 \lambda(e^0_1) m_{2 \lambda} \mu(x,t) }{r(\widehat{u}(x,t))^2}  \left( \frac{\widehat{u}'(x,t)}{\tan (2 \lambda(e^0_1) x)} \frac{r'(\widehat{u}(x,t))}{r(\widehat{u}(x,t))} + \frac{\lambda(e^0_1)}{\sin^2 ( 2 \lambda(e^0_1) x)} \right),
\end{aligned}
\end{equation}
where $\beta(x,t)$ is a certain smooth function. 
We have
\begin{align*}
&-\frac{2 \lambda(e^0_1) m_{\lambda}\mu(x,t) }{r(\widehat{u}(x,t))^2}  \left( \frac{2 \widehat{u}'(x,t)}{\tan (\lambda(e^0_1) x)} \frac{r'(\widehat{u}(x,t))}{r(\widehat{u}(x,t))} + \frac{\lambda(e^0_1)}{\sin^2 (\lambda(e^0_1) x)} \right) \\
&= -\frac{2 \lambda(e^0_1) m_{\lambda}\mu(x,t) }{r(\widehat{u}(x,t))^2}  \left( \frac{\widehat{u}'(x,t)}{\sqrt{\lambda(e^0_1)}} \frac{r'(\widehat{u}(x,t))}{r(\widehat{u}(x,t))}  \cos(\lambda(e^0_1) x)+ \frac{\sqrt{\lambda(e^0_1)}}{\sin (\lambda(e^0_1) x)} \right)^2 \\
&+ \frac{2 \lambda(e^0_1) m_{\lambda}\mu(x,t) }{r(\widehat{u}(x,t))^2} \frac{\mu(x,t)}{\lambda(e^0_1)} \frac{r'(\widehat{u}(x,t))^2}{r(\widehat{u}(x,t))^2}  \cos^2(\lambda(e^0_1) x) \\
&\leq \frac{2 m_{\lambda}\mu(x,t)^2 }{r(\widehat{u}(x,t))^2} \frac{r'(\widehat{u}(x,t))^2}{r(\widehat{u}(x,t))^2},
\end{align*}
and by the same computation, we also have
\begin{align*}
-\frac{8 \lambda(e^0_1) m_{2 \lambda} \mu(x,t) }{r(\widehat{u}(x,t))^2}  \left( \frac{\widehat{u}'(x,t)}{\tan (2 \lambda(e^0_1) x)} \frac{r'(\widehat{u}(x,t))}{r(\widehat{u}(x,t))} + \frac{\lambda(e^0_1)}{\sin^2 ( 2 \lambda(e^0_1) x)} \right)
\leq \frac{2 m_{2 \lambda} \mu(x,t)^2 }{r(\widehat{u}(x,t))^2} \frac{r'(\widehat{u}(x,t))^2}{r(\widehat{u}(x,t))^2}.
\end{align*}
Since $M$ is a rank one symmetric space, we have
\begin{align*}
&\frac{4 r'(\widehat{u}(x,t))^2 \mu(x,t)^2 }{ ( r(\widehat{u}(x,t))^2 + \mu(x,t))^2} 
+ \frac{2 m_{\lambda}\mu(x,t)^2 }{r(\widehat{u}(x,t))^2} \frac{r'(\widehat{u}(x,t))^2}{r(\widehat{u}(x,t))^2}
+ \frac{2 m_{2 \lambda} \mu(x,t)^2 }{r(\widehat{u}(x,t))^2} \frac{r'(\widehat{u}(x,t))^2}{r(\widehat{u}(x,t))^2} \\
\leq& \frac{4 \mu(x,t)^2 }{r(\widehat{u}(x,t))^2} \frac{r'(\widehat{u}(x,t))^2}{r(\widehat{u}(x,t))^2} + \frac{2 (n - 1) \mu(x,t)^2 }{r(\widehat{u}(x,t))^2} \frac{r'(\widehat{u}(x,t))^2}{r(\widehat{u}(x,t))^2} \\
\leq& \frac{2 (n + 1) \mu(x,t)^2 }{r(\widehat{u}(x,t))^2} \frac{r'(\widehat{u}(x,t))^2}{r(\widehat{u}(x,t))^2}.
\end{align*}
Therefore, we obtain the desired inequality.
\end{proof}

The following lemma corresponds to Lemma 3.5 of \cite{HZZ}.
\begin{lemma}\label{lem:const}
Let $Z(t)$, $\overline{\phi}(t)$ be solutions to the following differential equations, respectively:
\begin{equation*}
\frac{d Z}{d t}(t) = - n \frac{r'(Z(t))}{r(Z(t))}, \quad Z(0) = z_0 \in (-a, a),
\end{equation*}
\begin{equation*}
\frac{d \overline{\phi}}{d t}(t) = 2 n \overline{\phi}(t) \frac{r'(Z(t))^2}{r(Z(t))^2}, \quad \overline{\phi}(0) = \overline{\phi}_0 > 0.
\end{equation*}
Then, we have $\overline{\phi}(t) r(Z(t))^2 = \overline{\phi}(0) r(Z(0))^2$ for all $t \in [0,\infty)$.
\end{lemma}
\begin{proof}
The statement of this lemma follows from $d(\overline{\phi}(t) r(Z(t))^2)/d t = 0$.
\end{proof}

\begin{lemma}\label{lem:main1}
Let $I = (-a, a)$ and $r \colon I \to \mathbb{R}$ be a positive smooth function. 
Assume that $r$ satisfies the following conditions:
\begin{description}
\item[(1)] $r(0) = 1$ and $r'(0) = 0$;
\item[(2)] $r'(z) > 0$ for all $z \in (0, a)$ and $r'(z) < 0$ for all $z \in (-a, 0)$;
\item[(3)] $r(z) r''(z) - r'(z)^2 \geq 0$.
\end{description} 
If $\widehat{u}_0$ satisfies 
\begin{equation*}
\max_{x \in [0, \pi/2 \lambda(e^0_1)]} \abs{\widehat{u}_0(x)} < a_0, \quad \max_{x \in [0, \pi/2 \lambda(e^0_1)]} \abs{\widehat{u}'_0(x)} < \sqrt{ \frac{n}{2(n+1) \log R_0} },
\end{equation*}
where $R_0 \coloneqq \max \{ r(a_0), r(-a_0) \}$, then there exists a positive constant $d$ such that  
\begin{equation*}
\max_{x \in [0, \pi/2 \lambda(e^0_1)]} \abs{\widehat{u}'(x,t)} < d
\end{equation*}
for all $t \in [0,\widehat{T})$ . 
\end{lemma}
\begin{proof}
By the assumptions, we have $r(z) \geq 1$ and $(r'/r)' \geq 0$. Hence, from Lemma \ref{lem:mu ineq}, we obtain
\begin{align}
\frac{\partial \mu}{\partial t}(x,t) 
\leq& \frac{\mu''(x,t)}{r(\widehat{u}(x,t))^2  + \mu(x,t)} + \beta(x,t) \mu'(x,t)  - 2 n \left( \frac{r'}{r} \right)' \hspace{-1.5mm} (\widehat{u}(x,t)) \mu(x,t) \notag\\ 
&- \left( \frac{r'}{r} \right)' \hspace{-1.5mm} (\widehat{u}(x,t)) \, \frac{ 2 \mu(x,t)^2 }{ r(\widehat{u}(x,t))^2 + \mu(x,t)} 
+  \frac{2 (n + 1) \mu(x,t)^2}{r(\widehat{u}(x,t))^2} \frac{r'(\widehat{u}(x,t))^2}{r(\widehat{u}(x,t))^2} \notag\\
\leq& \frac{\mu''(x,t)}{r(\widehat{u}(x,t))^2  + \mu(x,t)} + \beta(x,t) \mu'(x,t)  +2(n+1)\mu(x,t)^2 \frac{r'(\widehat{u}(x,t))^2}{r(\widehat{u}(x,t))^2}. \label{ineq:mu}
\end{align}
Define $Z^{\pm}(t)$ and $\overline{\phi}^{\pm}(t)$ by solutions to the following differential equations, respectively:
\begin{equation*}
\frac{d Z^{\pm}}{d t}(t) = - n \frac{r'(Z^{\pm}(t))}{r(Z^{\pm}(t))}, \quad Z^{\pm}(0) = \pm a_0,
\end{equation*}
\begin{equation*}
\frac{d \overline{\phi}^{\pm}}{d t}(t) = 2 n \overline{\phi}^{\pm}(t) \frac{r'(Z^{\pm}(t))^2}{r(Z^{\pm}(t))^2}, \quad \overline{\phi}^{\pm}(0) = \overline{\phi}_0 > 0.
\end{equation*}
Here, we note that $Z^{+}(t)$ and $Z^{-}(t)$ define mean curvature flows of graph hypersurfaces in $\overline{M}$.
Hence, by the comparison principle of the mean curvature flow, $Z^{-}(t) < \widehat{u}(x,t) < Z^{+}(t)$ holds for all $(x,t) \in [0,\pi/2 \lambda(e^0_1)] \times [0,\widehat{T})$. 
We also have 
\begin{equation*}
\frac{r'(\widehat{u}(x,t))^2}{r(\widehat{u}(x,t))^2} \leq 
\begin{cases}
\frac{r'(Z^{+}(t))^2}{r(Z^{+}(t))^2} & \text{if } -Z^{-}(t) \leq Z^{+}(t)\\
\frac{r'(Z^{-}(t))^2}{r(Z^{-}(t))^2} & \text{if }Z^{+}(t) \leq -Z^{-}(t)
\end{cases}
\end{equation*}
by the monotonicity of $r'/r$.
Set $\phi(t) \coloneqq \max_{x \in [0, \pi/2 \lambda(e^0_1)]} \mu(x,t)$. If $Z^{-}(t) \leq Z^{+}(t)$ holds for $t \in (t_1, t_2)$, then from \eqref{ineq:mu}, we have
\begin{equation*}
\frac{d \phi}{d t}(t) \leq 2(n + 1) \phi(t)^2 \frac{r'(Z^{+}(t))^2}{r(Z^{+}(t))^2} = \frac{n+1}{n} \frac{\phi(t)^2}{\overline{\phi}^{+}(t)} \frac{d \overline{\phi}^{+}}{d t}(t)
\end{equation*}
for $t \in (t_1, t_2)$.
Hence it holds that
\begin{equation*}
- \frac{d}{d t} \left( \frac{1}{\phi(t)} \right) \leq \frac{n+1}{n} \frac{d \log \overline{\phi}^{+}(t) }{d t},
\end{equation*}
then integrate it to obtain
\begin{align*}
\frac{1}{\phi(t)} - \frac{1}{\phi(t_2)} 
\leq& \frac{n+1}{n} \left( \log \overline{\phi}^{+}(t_2) - \log \overline{\phi}^{+}(t) \right) \\
=& \frac{n+1}{n} \left( \log \overline{\phi}^{+}(t) + 2 \log r(Z^{+}(t)) -2  \log r(Z^{+}(t_2)) - \log \overline{\phi}^{+}(t) \right) \\
\leq& \frac{2(n+1)}{n} \left( \log r(Z^{+}(t)) -  \log r(Z^{+}(t_2)) \right),
\end{align*}
where we use $\overline{\phi}^{+}(t_2) = \overline{\phi}^{+}(t)r(Z^{+}(t))^2/r(Z^{+}(t_2))^2$ by Lemma \ref{lem:const}. 
Similarly, if $Z^{+}(t) \leq -Z^{-}(t)$ holds for $t \in (t_2, t_3)$, we have
\begin{equation*}
\frac{1}{\phi(t_2)} - \frac{1}{\phi(t)} \leq \frac{2(n+1)}{n} \left( \log r(Z^{-}(t_2)) -  \log r(Z^{-}(t)) \right).
\end{equation*}
Thus, for $s \in (t_1, t_2)$ and $t \in (t_2, t_3)$, we obtain
\begin{equation*}
\frac{1}{\phi(s)} - \frac{1}{\phi(t)} \leq \frac{2(n+1)}{n} \left( \log r(Z^{+}(s)) -  \log r(Z^{-}(t)) \right) \leq \frac{2(n+1)}{n} \log r(Z^{+}(s)) 
\end{equation*}
since $Z^{+}(t_2) = -Z^{-}(t_2)$.
By the above argument, we obtain
\begin{align*}
\frac{1}{\phi(0)} - \frac{1}{\phi(t)} \leq \frac{2(n+1)}{n} \log R_0
\end{align*}
for $t \in [0,\widehat{T})$, where $R_0 \coloneqq \max \{ r(a_0), r(-a_0) \}$.
Thus, by the assumption, we have 
\begin{equation*}
0 < \frac{1}{\phi(0)} - \frac{2(n+1)}{n} \log R_0 \leq \frac{1}{\phi(t)},
\end{equation*}
and we conclude that
\begin{equation*}
\phi(t) \leq \frac{n \phi(0)}{n - 2 (n+1) \phi(0) \log R_0}.
\end{equation*}
\end{proof}

\begin{lemma}\label{lem:main2}
Let $I = (-\infty, a)$ and $r \colon I \to \mathbb{R}$ be a positive smooth function.
Assume that $r$ satisfies the following conditions:
\begin{description}
\item[(1)] $r'(z) > 0$ for all  $z \in (-\infty,a)$;
\item[(2)] $r(z) r''(z) - (1 + \alpha){r'(z)}^2 \geq 0$ for some positive constant $\alpha$ greater than one.
\end{description}
If $\widehat{u}_0$ satisfies that
\begin{equation*}
\min_{x \in [0, \pi/2 \lambda(e^0_1)]} \widehat{u}_0(x) > a_1, \quad \max_{x \in [0, \pi/2 \lambda(e^0_1)]} \abs{\widehat{u}'_0(x)} < \sqrt{\frac{n(\alpha - 1)}{n+1}} r(a_1)
\end{equation*}
then we have
\begin{equation*}
\max_{x \in [0, \pi/2 \lambda(e^0_1)]} \abs{\widehat{u}'(x,t)} < \sqrt{\frac{n(\alpha - 1)}{n+1}} r(Z(t))
\end{equation*}
for all $t \in [0,\widehat{T})$, where $Z(t)$ is a solution to the following differential equation,
\begin{equation*}
\frac{d Z}{d t}(t) = - n \frac{r'(Z(t))}{r(Z(t))}, \quad Z(0) = a_1.
\end{equation*} 
\end{lemma}
\begin{proof}
By the assumption, we have
\begin{equation*}
-\left( \frac{r'}{r} \right)'(z) \leq - \alpha \left( \frac{r'}{r} \right)^2 (z).
\end{equation*}
From Lemma \ref{lem:mu ineq}, we obtain
\begin{align*}
\frac{\partial \mu}{\partial t}(x,t) 
\leq& \frac{\mu''(x,t)}{r(\widehat{u}(x,t))^2  + \mu(x,t)} + \beta(x,t) \mu'(x,t)  - 2 n \alpha \frac{r'(\widehat{u}(x,t))^2}{r(\widehat{u}(x,t))^2} \mu(x,t) \\
&- \alpha \frac{r'(\widehat{u}(x,t))^2}{r(\widehat{u}(x,t))^2} \frac{ 2 \mu(x,t)^2 }{ r(\widehat{u}(x,t))^2 + \mu(x,t)} 
+  \frac{2 (n + 1)\mu(x,t)^2}{r(\widehat{u}(x,t))^2} \frac{r'(\widehat{u}(x,t))^2}{r(\widehat{u}(x,t))^2} \\
=& \frac{\mu''(x,t)}{r(\widehat{u}(x,t))^2  + \mu(x,t)} + \beta(x,t) \mu'(x,t) \\
&-2 \frac{r'(\widehat{u}(x,t))^2}{r(\widehat{u}(x,t))^2}  \mu(x,t) \left( n \alpha + \frac{ \alpha \mu(x,t)}{ r(\widehat{u}(x,t))^2 + \mu(x,t)} - (n+1) \frac{\mu(x,t)}{r(\widehat{u}(x,t))^2} \right).
\end{align*}
Set
\begin{equation*}
\psi(x,t) \coloneqq \frac{\mu(x,t)}{r(Z(t))^2},
\end{equation*}
where $Z$ is the solution to the following differential equaiton
\begin{equation*}
\frac{d Z}{d t}(t) = - n \frac{r'(Z(t))}{r(Z(t))}, \quad Z(0) = a_1.
\end{equation*}
Note that $Z(t) < \widehat{u}(x,t)$ holds for all $(x,t) \in [0,\pi/2 \lambda(e^0_1)] \times [0,\widehat{T})$ as in the proof of Lemma \ref{lem:main1}.
The function $\psi$ satisfies 
\begin{align*}
\frac{\partial \psi}{\partial t}(x,t) 
=& \frac{1}{r(Z(t))^2} \frac{\partial \mu}{\partial t}(x,t) -  \frac{2 \mu(x,t)}{r(Z(t))^3} r'(Z(t)) \frac{d Z}{d t}(t) \\
\leq& \frac{\psi''(x,t)}{r(\widehat{u}(x,t))^2  + \mu(x,t)} + \beta(x,t) \psi'(x,t) \\
&-2 \frac{r'(\widehat{u}(x,t))^2}{r(\widehat{u}(x,t))^2}  \psi(x,t) \left( n \alpha + \frac{ \alpha \mu(x,t)}{ r(\widehat{u}(x,t))^2 + \mu(x,t)} - (n+1) \frac{\mu(x,t)}{r(\widehat{u}(x,t))^2} \right) \\
&+  2 n \psi(x,t) \frac{r'(Z(t))^2}{r(Z(t))^2} \\
\leq& \frac{\psi''(x,t)}{r(\widehat{u}(x,t))^2  + \mu(x,t)} + \beta(x,t) \psi'(x,t) \\
&-2 \frac{r'(\widehat{u}(x,t))^2}{r(\widehat{u}(x,t))^2}  \psi(x,t) \left( n \alpha + \frac{ \alpha \mu(x,t)}{ r(\widehat{u}(x,t))^2 + \mu(x,t)} - (n+1) \psi(x,t) - n \right),
\end{align*}
where we use the monotonicity of $r$ and $r'/r$. 
Set $c \coloneqq n(\alpha - 1)/(n+1)$.
Suppose that the function $\psi$ satisfies $\max_{x \in [0,\pi/2\lambda(e^0_1)]} \psi(x,0) < c$ and that $\psi(x_0, t_0)$ attains $c$ for the first time. 
Then, the followings holds:
\begin{equation*}
\frac{\partial \psi}{\partial t}(x_0, t_0) \geq 0, \quad \psi''(x_0, t_0) \leq 0, \quad \psi'(x_0, t_0) = 0.
\end{equation*}
Hence, we have
\begin{equation*}
0 \leq -2 \frac{r'(\widehat{u}(x_0,t_0))^2}{r(\widehat{u}(x,t))^2}  c \left( n \alpha + \frac{ \alpha \mu(x_0,t_0)}{ r(\widehat{u}(x_0,t_0))^2 + \mu(x_0,t_0)} -  (n+1) c - n \right),
\end{equation*}
equivalently,
\begin{equation}\label{eq:condition}
( r(\widehat{u}(x_0,t_0))^2 + \mu(x_0,t_0) ) (n \alpha - (n+1) c - n)  + \alpha \mu(x_0,t_0) \leq 0.
\end{equation} 
However, from the assumption, we have
\begin{equation*}
n \alpha - (n+1) c - n = n(\alpha-1) - (n+1) c = 0.
\end{equation*}
This leads to a contradiction. 
Thus, we have 
\begin{equation*}
\psi(x, t) < \frac{n(\alpha - 1)}{n+1}
\end{equation*}
for all $(x,t) \in [0,\pi/2 \lambda(e^0_1)] \times [0,\widehat{T})$.
\end{proof}

\begin{proof}[Proof of Theorems \ref{thm:corHZZ} and \ref{thm:corF}]
Define a function $\Theta_t \colon M \to \mathbb{R}$ by
\begin{equation*}
\Theta_t(p) \coloneqq \overline{g}_{F_t(p)}((N_t)_p, (\overline{\partial_z})_{F_t(p)}),
\end{equation*}
where $N_t$ is a unit normal vector field of $F_t(M)$. 
Let $u_0 \colon M \to I$ be a $K$-invariant smooth function and set $F_t(p) = (p, u_t(p))$. 
Then, we have
\begin{equation}\label{eq:angle}
\Theta_t(p) = \frac{r(u_t(p))}{\sqrt{r(u_t(p))^2 + \abs{\mathrm{grad}\,  u_t(p)}^2}} = \frac{r(\widehat{u}_t(x))}{\sqrt{r(\widehat{u}_t(x))^2 + \widehat{u}'_t(x)^2}},
\end{equation}
where $p = \mathrm{exp}_{e K}\left( x e^0_1 \right)$.
From Lemma \ref{lem:main1} and Lemma \ref{lem:main2}, $(\widehat{u}'_t)^2$ is uniformly bounded under both assumptions of Theorems \ref{thm:corHZZ} and \ref{thm:corF}.
The uniform boundedness of $(\widehat{u}'_t)^2$ and equation \eqref{eq:angle} imply that there exists a positive constant $c > 0$ such that $\Theta_t \geq c$ for all $t$. 
Therefore, we can derive the long-time existence of the flow by the arguments as in the proofs of Theorem 1.2 in \cite{HZZ} and Theorem 1.2 in \cite{F}.
\end{proof}
\begin{rem}
From the initial condition of Theorem \ref{thm:corF}, we have  
\begin{equation*}
\max_{x \in [0, \pi/2 \lambda(e^0_1)]} \abs{\widehat{u}'_0(x)} < \sqrt{\frac{n(\alpha - 1)}{n+1}} r(a_1) \leq \sqrt{\frac{n(\alpha - 1)}{n+1}} r(\widehat{u}_0(x)).
\end{equation*}
This implies that 
\begin{equation*}
\Theta_0(p) = \frac{r(\widehat{u}_0(x))}{\sqrt{r(\widehat{u}_0(x))^2 + \widehat{u}'_0(x)^2}} \geq \frac{r(\widehat{u}_0(x))}{\sqrt{r(\widehat{u}_0(x))^2 + \frac{n(\alpha - 1)}{n+1}  r(\widehat{u}_0(x))^2 } } = \sqrt{\frac{n+1}{n+1 + n (\alpha -1) }} > \sqrt{\frac{1}{\alpha}}.
\end{equation*}
Therefore, the long-time existence of the flow follows directly from Theorem 1.2 of \cite{F}. 
However, the upper bound of $(\widehat{u}'_t)^2$ can be provided explicitely in this setting.
\end{rem}

\section*{Acknowledgements} 
The second named author is supported by JSPS KAKENHI Grant Number JP22K03300.

\vspace{0.5truecm}

\begin{flushright}
{\small 
Naotoshi Fujihara \\
Department of Mathematics \\
Graduate School of Science \\
Tokyo University of Science \\
1-3 Kagurazaka \\
Shinjuku-ku \\
Tokyo 162-8601 \\
Japan \\
(E-mail: 1123706@ed.tus.ac.jp)
}
\end{flushright}

\begin{flushright}
{\small 
Naoyuki Koike \\
Department of Mathematics \\
Faculty of Science \\
Tokyo University of Science \\
1-3 Kagurazaka \\
Shinjuku-ku \\
Tokyo 162-8601 \\
Japan \\
(E-mail: koike@rs.tus.ac.jp)
}
\end{flushright}


\begin{thebibliography}{99}
\bibitem{BM}
A. A. Borisenko, and V. Miquel. Mean curvature flow of graphs in warped products. Trans. Amer. Math. Soc. \textbf{364}(9) (2012) 4551-4587.
\bibitem{F}
N. Fujihara. Mean curvature flows of graphs sliding off to infinity in warped product manifolds. Differential Geom. Appl. \textbf{97} (2024), 102207.
\bibitem{HZZ}
Z. Huang, Z. Zhang, and H. Zhou. Mean curvature flows of closed hypersurfaces in warped product manifolds. Math. Res. Lett. \textbf{26}(5) (2019), 1393-1413.
\bibitem{O'Neill}
B. O'Neill. Semi-Riemannian geometry with applications to relativity. Pure Appl. Math. \textbf{103}, Academic Press, New York (1983).
\bibitem{PT}
R. S. Palais, and C. L. Terng. Critical Point Theory and Submanifold Geometry. Lecture Notes in Math. \textbf{1353}, Springer-Verlag, Berlin (1988).
\bibitem{V}
L. Verh\'{o}czki. Shape operators of orbits of isotropy subgroups in Riemannian symmetric spaces of the compact type. Beitr\"age Algebra Geom. \textbf{36}(1995), 155-170.
\end{thebibliography}
\end{document}